\documentclass{amsart}

\usepackage{comment}

\usepackage[mathscr]{eucal}
\usepackage{amssymb}
\usepackage{graphicx}
\usepackage{psfrag} 
\usepackage{epstopdf} 
\usepackage[usenames,dvipsnames]{color}
\usepackage[normalem]{ulem}
\usepackage{amsthm}
\usepackage{bbm}
\usepackage{enumerate}
\usepackage{array}
\usepackage{amsmath}

\usepackage{hyperref}
\usepackage[T1]{fontenc}

\hypersetup{colorlinks=true,linkcolor={Brown},citecolor={Brown},urlcolor={Brown}}

\usepackage{cleveref}

\numberwithin{equation}{section}
\makeindex

\usepackage[all]{xy}
\CompileMatrices

\newdir{ >}{{}*!/-10pt/\dir{>}}

\swapnumbers 

\newtheorem{Thm}[equation]{Theorem}
\newtheorem*{Thm*}{Theorem}
\newtheorem{Prop}[equation]{Proposition}
\newtheorem{Lem}[equation]{Lemma}
\newtheorem{Cor}[equation]{Corollary}
\theoremstyle{remark}
\newtheorem{Def}[equation]{Definition}

\newtheorem{Not}[equation]{Notation}

\newtheorem*{Conv*}{Conventions}

\newtheorem{Rem}[equation]{Remark}

\newcommand{\nc}{\newcommand}
\nc{\dmo}{\DeclareMathOperator}

\dmo{\Ab}{Ab}
\dmo{\AbMon}{AbMon}
\dmo{\Abelem}{Abelem}
\dmo{\Aut}{Aut}
\dmo{\Bi}{bi}
\dmo{\Bisets}{Bisets}
\dmo{\DER}{\mathsf{DER}}
\dmo{\ADDER}{\mathsf{ADDER}}
\dmo{\coev}{coev}
\dmo{\Coloc}{Coloc}
\dmo{\ev}{ev}
\dmo{\Fib}{Fib}
\dmo{\Free}{Free}
\dmo{\Id}{Id}
\dmo{\Loc}{Loc}
\dmo{\rmI}{I}
\dmo{\rmL}{L}
\dmo{\rmR}{R}
\dmo{\Spc}{Spc}
\dmo{\Thick}{Thick}
\dmo{\chara}{char}%
\dmo{\coh}{coh} 
\dmo{\Coind}{CoInd}
\dmo{\coker}{coker}

\dmo{\cone}{cone}
\dmo{\Der}{D}
\dmo{\Ch}{Ch}
\nc{\Rder}{\mathrm{R}} 
\nc{\Lder}{\mathrm{L}} 
\dmo{\Khocat}{K}
\dmo{\End}{End}
\dmo{\Ext}{Ext}
\dmo{\rmH}{H}
\dmo{\Ho}{Ho}
\dmo{\Hom}{Hom}
\dmo{\id}{id}
\dmo{\Img}{Im}
\dmo{\incl}{incl}
\dmo{\Ind}{Ind}
\dmo{\CoInd}{CoInd}

\dmo{\Ker}{Ker}
\dmo{\Les}{Les}
\dmo{\Map}{Map}%
\dmo{\Mod}{Mod}
\dmo{\GrMod}{GrMod}
\dmo{\lax}{lax}
\dmo{\modname}{mod}%
\dmo{\grmod}{grmod}
\dmo{\Mor}{Mor}%
\dmo{\Obj}{Obj}
\dmo{\Or}{Or}
\dmo{\Oldorbit}{\mathcal{O}} 
\dmo{\Fix}{Fix} 
\dmo{\Ev}{Ev} 
\dmo{\pr}{pr}
\dmo{\canin}{in} 
\dmo{\Proj}{Proj} 
\dmo{\Inj}{Inj} 
\dmo{\proj}{proj}
\dmo{\Qcoh}{Qcoh}
\dmo{\rank}{rank}
\dmo{\Res}{Res}
\dmo{\Defl}{Def}
\dmo{\Infl}{Inf}
\dmo{\Iso}{Iso}
\dmo{\Rname}{R}
\dmo{\Sp}{Sp} 
\dmo{\SH}{SH}
\nc{\SHp}{\SH_{(p)}}
\dmo{\smallb}{b}
\dmo{\smallperf}{perf}
\dmo{\Spec}{Spec}
\dmo{\Spech}{Spec^h}
\dmo{\Stab}{Stab}
\dmo{\stab}{stab}
\dmo{\supp}{supp}
\dmo{\switch}{switch}
\dmo{\TTR}{TTR}
\dmo{\Spanname}{{\sf Span}}
\dmo{\map}{map}
\dmo{\Rel}{Rel}

\nc{\Ivo}[1]{{\color{OliveGreen}#1}}
\nc{\Paul}[1]{{\color{Blue}#1}}
\nc{\Pout}[1]{\Paul{\sout{#1}}}
\nc{\Iout}[1]{\Ivo{\sout{#1}}}
\nc{\Prule}{\Paul{\smallbreak\hrule\smallbreak}}

\nc{\SEcell}{\rotatebox[origin=c]{45}{$\Downarrow$}} 
\nc{\NEcell}{\rotatebox[origin=c]{135}{$\Downarrow$}} 
\nc{\SWcell}{\rotatebox[origin=c]{-45}{$\Downarrow$}} 
\nc{\NWcell}{\rotatebox[origin=c]{-135}{$\Downarrow$}} 
\nc{\Scell}{\rotatebox[origin=c]{0}{$\Downarrow$}} 
\nc{\Ncell}{\rotatebox[origin=c]{0}{$\Uparrow$}} 
\nc{\Ecell}{\Rightarrow}
\nc{\oEcell}[1]{\overset{\scriptstyle #1}{\Ecell}}
\nc{\isoEcell}{\overset{\sim}{\,\Ecell\,}}
\nc{\isocell}[1]{\undersett{ #1}\isoEcell}
\nc{\Isocell}[1]{\undersett{ #1}{\overset{\sim}{\Longrightarrow}}}
\nc{\Wcell}{\rotatebox[origin=c]{90}{$\Uparrow$}} 
\nc{\Span}{\Spanname}
\nc{\Spanhat}{\textrm{\sf S}\widehat{\textrm{\sf pan}}} 
\nc{\tSpan}{\pih{\Spanname}}
\nc{\IFF}{$\Leftrightarrow$}
\nc{\ass}{\mathrm{ass}} 
\nc{\lun}{\mathrm{lun}} 
\nc{\run}{\mathrm{run}} 
\nc{\fun}{\mathrm{fun}} 
\nc{\un}{\mathrm{un}} 
\nc{\Crich}{\underline{\cat{C}}}
\nc{\DbG}{\Db(\kk G\mmod)}
\nc{\uA}{\underline{A}}
\nc{\doublequot}[3]{#1\backslash #2/#3}
\nc{\HGK}{\doublequot HGK}
\nc{\quadtext}[1]{\quad\textrm{#1}\quad}
\nc{\qquadtext}[1]{\qquad\textrm{#1}\qquad}
\nc{\PZG}{\cat{C}_{\bbZ}(\bbZ G)}
\nc{\TTRK}{\TTR(\cat K)}
\nc{\psets}{\mathsf{-sets}_\sbull}
\nc{\Gsets}{G\mathsf{-sets}}
\nc{\Hsets}{H\mathsf{-sets}}
\nc{\AddK}{\Add^{\Sigma}(\cat K)}
\nc{\adj}{\dashv\,}
\nc{\adjto}{\rightleftarrows}
\nc{\AK}{A\MModcat{K}}
\nc{\BK}{B\MModcat{K}}
\nc{\bbA}{\mathbb{A}}
\nc{\bbB}{\mathbb{B}}
\nc{\bbC}{\mathbb{C}}
\nc{\bbD}{\mathbb{D}}
\nc{\bbF}{\mathbb{F}}
\nc{\bbI}{\mathbb{I}}
\nc{\bbM}{\mathbb{M}}
\nc{\bbN}{\mathbb{N}}
\nc{\bbP}{\mathbb{P}}
\nc{\bbQ}{\mathbb{Q}}
\nc{\bbR}{\mathbb{R}}
\nc{\bbZ}{\mathbb{Z}}
\nc{\bbZp}{\mathbb{Z}_{(p)}}
\nc{\Sphere}{\mathbb{S}} 
\nc{\cat}[1]{\mathcal{#1}}
\nc{\Displ}{\displaystyle}
\nc{\ie}{{\sl i.e.}\ }
\nc{\cf}{{\sl cf.}\ }
\nc{\into}{\mathop{\rightarrowtail}}
\nc{\inv}{^{-1}}
\nc{\isoto}{\buildrel \sim\over\to}
\nc{\isotoo}{\mathop{\buildrel \sim\over\too}}
\nc{\kk}{\Bbbk}
\nc{\onto}{\mathop{\twoheadrightarrow}}
\nc{\too}{\mathop{\longrightarrow}\limits}
\nc{\xytriangle}[7]{\xymatrix@C=#7em{#1\ar[r]^-{\Displ #4} & #2 \ar[r]^-{\Displ #5}&#3\ar[r]^-{\Displ #6}&T #1}}
\nc{\ababs}{{\sl ab absurdo}}
\nc{\adh}[1]{\overline{#1}}
\nc{\adhoc}{{\sl ad hoc}}
\nc{\adhpt}[1]{\adh{\{#1\}}}
\nc{\afortiori}{{\sl a fortiori}}
\nc{\aka}{{a.\,k.\,a.}\ }
\nc{\ala}{{\sl \`a la}\ }
\nc{\apriori}{{\sl a priori}}
\nc{\Autcat}[1]{\Aut_{\cat #1}}
\nc{\cO}{\mathcal{O}}
\nc{\calO}{\mathcal{O}}
\nc{\cV}{\mathcal{V}}
\nc{\Db}{\Der^{\smallb}}
\nc{\Dqc}{\Der_{\Qcoh}}
\nc{\Dperf}{\Der^{\smallperf}}
\nc{\eg}{{\sl e.g.}}
\nc{\eps}{\varepsilon}
\nc{\equalby}[1]{\overset{\textrm{#1}}{=}}
\nc{\FFree}{\,\text{--}\Free}%
\nc{\FFreecat}[1]{\FFree_{\cat #1}}
\nc{\FK}{\mathcal{F}(\cat K)}
\nc{\gm}{\mathfrak{m}}
\nc{\Homcat}[1]{\Hom_{\cat #1}}
\nc{\Morcat}[1]{\Mor_{\cat #1}}
\nc{\hook}{\hookrightarrow}
\newcommand{\hooklongrightarrow}{\lhook\joinrel\longrightarrow}

\nc{\Idcat}[1]{\Id_{\cat{#1}}}
\nc{\ideal}[1]{\langle #1\rangle}
\nc{\ihom}{{\mathsf{hom}}} 
\nc{\ihomcat}[1]{\ihom_{\cat #1}}
\nc{\Kcat}[1]{#1\MModcat{K}}
\nc{\KP}{\cat{K}_{\cat P}}
\nc{\loccit}{{\sl loc.\ cit.}}
\nc{\lind}{\rmL\!}
\nc{\RR}{\rmR\!}
\nc{\Lotimes}{\otimes^{\rmL}}
\nc{\Mid}{\,\bigm|\,}
\nc{\MMod}{\,\text{-}\Mod}%
\nc{\Exact}{\mathfrak K\,\text{-}\mathrm{Exa}^\mathbb Z/2_\infty} 
\nc{\MModcat}[1]{\MMod_{\cat #1}}%
\nc{\mmod}{\,\text{--}\modname}%
\nc{\mmodb}{\mmod^\sbull}%
\nc{\op}{{\mathrm{op}}}
\nc{\co}{{\mathrm{co}}}
\nc{\costar}{**}
\nc{\oto}[1]{\overset{#1}\to}
\nc{\ointo}[1]{\overset{#1}{\rightarrowtail}}
\nc{\loto}[1]{\overset{#1}{\leftarrow}}
\nc{\otoo}[1]{\overset{#1}{\,\longrightarrow\,}}
\nc{\lotoo}[1]{\overset{#1}{\,\longleftarrow\,}}
\nc{\ourfrac}[2]{\genfrac{}{}{0pt}{}{\Displ #1}{\scriptstyle #2}}
\nc{\ouriff}{\Leftrightarrow}
\nc{\oursetminus}{\!\smallsetminus\!}
\nc{\potimes}[1]{^{\otimes #1}}
\nc{\pproj}{\,\text{-}\proj}
\nc{\ptimes}[1]{^{\times #1}}
\nc{\dd}[1]{_{{\scriptscriptstyle(#1)}}}
\nc{\uu}[1]{^{{\scriptscriptstyle(#1)}}}
\nc{\pushout}{\textrm{\rm p.o.}}
\nc{\qp}{q_{_{\scriptstyle \cat P}}\!}%
\nc{\Rcat}[1]{\Rname_{\cat #1}^\sbull}
\nc{\rdto}{}
\nc{\restr}[1]{{|_{\scriptstyle #1}}}
\nc{\RK}{\Rcat{K}}
\nc{\sbull}{{\scriptscriptstyle\bullet}}
\nc{\SET}[2]{\bigl\{\,#1\Mid#2\,\bigr\}}
\nc{\SHA}{\SH{}^{\bbA^{1}}}
\nc{\SHfin}{\SH^{\text{\rm fin}}}
\nc{\smat}[1]{\left(\begin{smallmatrix} #1 \end{smallmatrix}\right)}
\nc{\SpcAK}{\Spc(A\MModcat{K})}
\nc{\SpcK}{\Spc(\cat K)}
\nc{\suppcat}[1]{\supp(\cat #1)}
\nc{\then}{\Rightarrow}
\nc{\tideal}[1]{\ideal{#1}}
\nc{\unit}{\mathbbm{1}}
\nc{\unitcat}[1]{\unit_{\cat #1}}
\nc{\vcorrect}[1]{{\vphantom{\vbox to #1em{}}}}
\nc{\onept}{\mathrm{B}} 
\nc{\undersett}[1]{\underset{\scriptstyle #1}}
\nc\noloc{\nobreak\mspace{6mu plus 1mu}{:}\nonscript\mkern-\thinmuskip\mathpunct{}\mspace{2mu}}

\nc{\HG}{\!{}^{^H}\overline{G}}
\nc{\uY}{\widetilde{Y}}

\nc{\ADD}{\mathsf{ADD}}
\nc{\MONADD}{\mathsf{MONADD}}
\nc{\SMONADD}{\mathsf{SMONADD}}
\nc{\Add}{\mathsf{Add}}
\nc{\SAD}{\mathsf{SAD}}
\nc{\Sad}{\mathsf{Sad}}
\nc{\Cat}{\mathsf{Cat}}
\nc{\CCat}{\textsf{-}\mathsf{Cat}}
\nc{\CAT}{\mathsf{CAT}}
\nc{\Dk}{\dual_{\kappa}}
\nc{\Dkk}{\dual_{\kappa'}}
\nc{\bs}{\backslash}
\nc{\biCpt}{\mathrm{biCpt}}
\nc{\biLCpt}{\mathrm{biLCpt}} 
\nc{\Groupoid}{\mathsf{Groupoid}}
\nc{\groupoid}{\mathsf{gpd}}
\nc{\gpd}{\groupoid}
\nc{\faithful}{\mathsf{faithful}}
\nc{\faith}{\mathsf{faithf}}
\nc{\exact}{\mathsf{ex}}
\nc{\smallfaithful}{\mathsf{f}}
\nc{\smallfused}{\mathsf{fus}}
\nc{\groupoidf}{\groupoid{}^{\smallfaithful}}
\nc{\groconn}{\groupoid_{\mathsf{conn}}}
\nc{\gps}{\mathsf{groups}} 
\nc{\group}{\mathsf{group}} 
\nc{\groupshort}{\mathsf{gr}}
\nc{\gpdG}{{\groupoidf_{\!\smallslash\!G}}} 
\nc{\GinG}{{\groupoidf_{G}}}
\nc{\gpdGfuz}{{\groupoid^{\smallfused}_{\!\smallslash\!G}}} 
\nc{\spanG}{{\widehat{\mathsf{gp}\,\,}\!\!\mathsf{d}}{}^\smallfaithful_{\!{}^{\scriptscriptstyle/}\!G}}
\nc{\biset}{\mathsf{biset}} 
\nc{\rfree}{\mathsf{rf}} 
\nc{\bifree}{\mathsf{bif}} 
\nc{\conj}{\mathsf{conj}} 
\nc{\smallslash}{{}^{\scriptscriptstyle/}}
\nc{\smallbs}{{}^{\scriptscriptstyle\backslash}}
\nc{\doublebs}{\smallbs\!\smallbs}

\nc{\Set}{\mathsf{Set}}
\nc{\set}{\mathsf{set}} 
\nc{\sset}{\textrm{-}\set}
\nc{\ssetfused}{\textrm{-}\underline{\set}} 
\nc{\ssetfuz}{\sset^{\smallfused}} 
\nc{\Top}{\mathrm{Top}} 
\nc{\Comp}{\mathsf{Top}^{\mathsf{comp}}}
\nc{\pih}[1]{\tau_{1}#1}
\nc{\all}{\mathsf{all}}

\dmo{\Fun}{\mathrm{Fun}} 
\dmo{\PsFun}{\mathsf{PsFun}} 
\dmo{\PsFunlax}{\mathsf{PsFun}_{\mathsf{lax}}}
\dmo{\PsFunoplax}{\mathsf{PsFun}_{\mathsf{oplax}}}
\dmo{\BCDex}{\mathsf{BCDex}_{\II\mathsf{-str}}}
\dmo{\BCDexdex}{\mathsf{BCDex}_{\II\mathsf{-dex}}}
\dmo{\biMack}{\mathsf{Mack}} 
\dmo{\Mackey}{Mack} 
\dmo{\twoFun}{2\mathsf{Fun}}

\nc{\Muniv}{\cat{M}^{\mathsf{univ}}}

\nc{\lG}{{}_{{\color{Gray}\scriptscriptstyle G}}}
\nc{\lH}{{}_{{\color{Gray}\scriptscriptstyle H}}}
\nc{\rG}{_{{\color{Gray}\!\scriptscriptstyle G}}}
\nc{\rH}{_{{\color{Gray}\!\scriptscriptstyle H}}}
\nc{\rK}{_{{\color{Gray}\!\scriptscriptstyle K}}}

\nc{\dual}{\Delta}

\nc{\ra}{\rightarrow}
\nc{\xra}{\xrightarrow}
\nc{\lto}{\leftarrow}
\nc{\olto}[1]{\overset{#1}\lto}

\nc{\C}{\mathbb{C}} 
\nc{\Cont}{\mathrm{C}} 
\nc{\Rep}{\mathrm{R}} 
\nc{\KK}{\mathrm{KK}} 
\nc{\Kth}{\mathrm{K}} 
\nc{\Cell}{\mathrm{Cell}}
\nc{\Modules}{\mathsf{Mod}}
\nc{\Alg}{\mathsf{Alg}}
\nc{\Sep}{\mathsf{Sep}}
\nc{\BurnG}{\cat{A}(G)}

\nc{\EndHere}{\bibliographystyle{alpha}\bibliography{articles}\printindex\end{document}}
\nc{\LimitOfCivilization}{\Paul{\goodbreak\hrule\smallbreak This is the limit of civilization.  Beyond this point, the right-hand 2-cells might be running upside-down.  $<$:-] \smallbreak\hrule\goodbreak}}

\nc{\cD}{\cat{D}}
\nc{\cG}{\cat{G}}
\nc{\cM}{\cat{M}}
\nc{\cN}{\cat{N}}

\nc{\DD}{\cat{D}}
\nc{\MM}{\cat{M}}
\nc{\NN}{\cat{N}}
\nc{\GG}{\mathbb{G}}
\nc{\II}{\mathbb{J}}
\nc{\JJ}{\II}
\nc{\AAA}{\mathbb{A}}
\nc{\leta}{{}^{\ell}\eta}
\nc{\reta}{{}^{r\!}\eta}
\nc{\leps}{{}^{\ell}\eps}
\nc{\reps}{{}^{r\!}\eps}

\nc{\gammap}[1]{\gamma^{(#1)}}
\nc{\what}[1]{\widehat{\cat{#1}}}

\nc{\und}[1]{{\kern1pt\underline{\kern-1pt{#1}\kern-1.5pt}\kern1.5pt}}

\nc{\Funplus}{\Fun_{+}}
\nc{\Mack}[1]{(Mack\,\ref{Mack-#1})}

\begin{document}


\title[The spectrum of equivariant KK-theory]{The spectrum of equivariant Kasparov theory for cyclic groups of prime order}
\author{Ivo Dell'Ambrogio}
\author{Ralf Meyer}
\date{\today}

\address{
\noindent   Univ.\ Lille, CNRS, UMR 8524 - Laboratoire Paul Painlev\'e, F-59000 Lille, France
}
\email{ivo.dell-ambrogio@univ-lille.fr}
\urladdr{http://math.univ-lille1.fr/$\sim$dellambr}
\address{
\noindent Mathematisches Institut, Universit\"at G\"ottingen, Bunsenstra{\ss}e 3--5, 37073 G\"ottingen, Germany
}
\email{ralf.meyer@mathematik.uni-goettingen.de}
\urladdr{http://www.uni-math.gwdg.de/rameyer/}

\begin{abstract}
  We compute the Balmer spectrum of the equivariant bootstrap
  category of separable $G$-C*-algebras when~$G$ is a
  group of prime order.
\end{abstract}

\subjclass[2010]{
19K35, 
19L47, 
18G80
}
\keywords{Equivariant Kasparov theory, triangulated categories, spectrum}

\thanks{First author partially supported by the Labex CEMPI (ANR-11-LABX-0007-01)}

\maketitle


\section{Introduction and results}
\label{sec:intro}%
\medskip

Let $G$ be a second countable locally compact group.  Kasparov's
$G$-equivariant KK-theory of complex separable $G$-C*-algebras
\cite{Kasparov88} defines a tensor triangulated category~$\KK^G$.
This fact was first recognized by Meyer and Nest~\cite{MeyerNest06},
who used it reformulate the Baum--Connes assembly map elegantly and
study its functorial properties.  The spectrum of an (essentially
small) tensor triangulated category~$\cat T$ is a certain
topological space $\Spc(\cat T)$.  This important invariant was
introduced by Balmer~\cite{Balmer05a}.  In \cite{DellAmbrogio10}, it
was observed that the Baum--Connes conjecture would follow from the
(also conjectural) surjectivity of the canonical map
$\bigsqcup_H\Spc(\KK^H)\to \Spc(\KK^G)$, where~$H$ runs through the
compact subgroups of~$G$.

Unfortunately, a geometric result such as the above surjectivity
still appears well out of reach; this is despite recent major
advances of tensor triangular geometry, the theory and computational
techniques concerned with the spectrum of tensor triangulated
categories (see the survey~\cite{BalmerICM}).  While numerous
spectra have been computed in fields ranging from topology to
modular representation theory, algebraic geometry and motivic theory
(see \cite{Balmer19} for a large catalogue), in noncommutative
geometry we are still striving to understand the most basic
examples.  The most general fact known to date is the existence,
when $G$ is a compact Lie group, of a canonical continuous
surjective map $\Spc \KK^G\to \Spec \Rep(G)$ onto the Zariski
spectrum of the complex character ring~$\Rep(G)$; and this holds for
rather formal reasons (see \cite[Cor.\,8.8]{Balmer10b}).  In order
to obtain sharper results, one must severely restrict the kind of
groups and algebras under study.

The present article contributes the first complete computation of a
truly equivariant example in this context, as well as techniques
which may prove useful in other cases.  Because of its good
generation properties, we consider as in~\cite{DellAmbrogio14} the
subcategory $\Cell(G)\subset \KK^G$ of \emph{$G$-cell algebras},
that is, the localizing subcategory of $\KK^G$ generated by the
function algebras $\Cont(G/H)$ for $H\leq G$; and we aim at
computing the spectrum of the tensor triangulated subcategory
$\Cell(G)^c$ of \emph{compact} objects (see \Cref{Rem:cpts}).  In
topology, this is analogous to considering the spectrum of the
equivariant stable homotopy category $\SH(G)^c$ of finite
$G$-spectra, as done in~\cite{BalmerSanders17}.  Moreover, we
restrict attention to finite groups~$G$.  In this case there is also
a continuous surjection
\[
  \rho_G\colon \Spc \Cell(G)^c \to \Spec \Rep(G),
\]
which is known to admit a canonical continuous splitting (see
\cite[Thm.\,1.4]{DellAmbrogio10}).  This map was conjectured by the
first author to be invertible.

The main result of the present paper settles some cases of the conjecture:

\begin{Thm}[See \Cref{sec:spc}]
  \label{Thm:main}
  Let $G= \mathbb Z/p\mathbb Z$ for a prime number~\(p\).  Then the
  canonical comparison map is a homeomorphism
  $\rho_G\colon \Spc \Cell(G)^c \overset{\sim}{\rightarrow} \Spec
  \Rep(G)$.
\end{Thm}

The representation ring $\Rep(G)$ of \(\mathbb Z/p\) is isomorphic
to the group ring:
\[
  \Rep(G) \cong \mathbb Z[G]\cong \mathbb Z[x]/(x^p-1).
\]
A \emph{thick tensor-ideal subcategory} of $\Cell(G)^c$ is a full
subcategory~$\cat C$ that is closed under taking isomorphic --~that
is, equivariantly KK-equivalent~-- objects, mapping cones,
suspensions, retracts and tensor products with arbitrary objects.  A
subset of a topological space is called \emph{specialization closed}
if it is a union $S= \bigcup_i Z_i$ of closed subsets~$Z_i$.

Since $\Cell(G)^c$ is rigid and idempotent complete by
\Cref{Rem:cpts}, \Cref{Thm:main} may be combined with general
abstract results (namely, \cite[Thm.\,14, Rmks.\,12
and~23]{BalmerICM}) to classify the thick tensor-ideal subcategories
of $\Cell(G)^c$:

\begin{Cor}
  \label{Cor:thick-class}
  If $G\cong \mathbb Z/ p\mathbb Z$, then there is a canonical
  bijection between
  \begin{enumerate}[\rm(1)]
  \item thick tensor-ideal subcategories~$\cat C$ of $\Cell(G)^c$ and
  \item specialization closed subsets~$S$ of the Zariski spectrum
    $\Spec (\mathbb Z[x]/(x^p-1))$.
  \end{enumerate}
\end{Cor}

The classification assigns to a thick tensor-ideal $\cat C$ the union of the \emph{supports} $\supp(A)$ of all algebras $A$ in~$\cat C$, and to a specialization closed subset $S$ the subcategory of all $A$ with $\supp(A)\subseteq S$.
Morally, the support of an object is the subset of $\Spec (\mathbb
Z[x]/(x^p-1))$ on which it `lives' (see \Cref{sec:ttg} for the
abstract definition and \cite[\S6]{DellAmbrogio10} for a more
concrete description).  
As a consequence of the corollary, a compact $G$-cell algebra $A$ can be built from another one $B$ using the tensor-triangular operations if and only if $\supp (A) \subseteq \supp (B)$.

If $G=1$ is the trivial group, then $\Cell(1)\subset \KK$ is just the
usual bootstrap category and $\Cell(1)^c$ consists of the
C*-algebras in the bootstrap category with finitely generated
K-theory groups.  In this case, our results are already known (see
\cite[Thm.\,1.2]{DellAmbrogio10} and also~\cite{DellAmbrogio11}).

\Cref{Thm:main} is significantly harder to prove than the case of
the trivial group and requires new ingredients.  Our proof strategy
is roughly inspired by that of~\cite{BalmerSanders17}, which
computes the spectrum of the above-mentioned equivariant stable
homotopy category~$\SH(G)^c$.  Specifically, we divide the problem
in two halves (see the beginning of \Cref{sec:spc}) and settle the
first half thanks to the separable monadicity of restriction
functors proved in \cite{BalmerDellAmbrogioSanders15}.  For the
second half, however, our proofs diverge because of fundamental
structural differences between equivariant stable homotopy and
KK-theory (see Remarks~\ref{Rem:cf-SH1} and \ref{Rem:cf-SH2}).  The
point is to compute the spectrum of a certain localization
$\cat Q(G)$ of~$\Cell(G)$ (see \Cref{sec:Q(G)}).  In order to do
this, we show that the tensor triangulated category $\cat Q(G)$ is
`weakly regular', so that we may apply the results
of~\cite{DellAmbrogioStanley16}.  (Incidentally, this method
\emph{cannot} be applied to the equivariant stable homotopy category
by \Cref{Rem:cf-SH2}.)  The weak regularity of $\cat Q(G)$ is shown
using our last crucial ingredient, K\"ohler's universal coefficient
theorem~\cite{Koehler10}.

This strategy may well provide a proof of the bijectivity of
$\rho_G$ for more general finite groups.  K\"ohler's result,
however, only holds for groups of prime order and hence would need
to be replaced by a more general argument.

\section{The equivariant bootstrap category}
\label{sec:prelim}%
\medskip

We collect here some structural results on equivariant KK-theory and cell algebras.
Our sources are \cite{MeyerNest06} and~\cite{DellAmbrogio14}.

Let $G$ be a finite group (for simplicity, as many of the statements
hold much more generally).  First of all, recall from
\cite{MeyerNest06} that the \emph{Kasparov category} of separable
$G$-C*-algebras, $\KK^G$, is a \emph{tensor triangulated category},
that is, it is triangulated in the sense of Verdier and carries a
symmetric monoidal structure.  The exact tensor functor is the
minimal tensor product of C*-algebras, equipped with the diagonal
group action.  The tensor unit object $\unit$ is the
algebra~\(\mathbb C\) of complex numbers with the trivial
$G$-action.  The category $\KK^G$ is essentially small and has all
countable coproducts, provided by \(\Cont_0\)-direct sums.  The
suspension functor $\Sigma={\Cont_0(\mathbb R)\otimes -}$ of the
triangulated structure of $\KK^G$ is 2-periodic, that is, there is a
natural isomorphism $\Sigma^2 \cong \mathrm{id}_{\KK^G}$, thanks to
the \emph{Bott isomorphism}
$\beta\colon \mathbb C\overset{\sim}{\to} \Cont_0(\mathbb R^2)$.
Let $\Rep(G)$ be the representation ring of~$G$.  The graded
endomorphism ring of the unit is
\[
\End_{\KK^G}(\unit)_* \cong \Rep(G) [\beta^{\pm 1}].
\]

\begin{Def}[{\cite{DellAmbrogio14}}]
  \label{Def:Cell}
  Let
  \[ \Cell(G):= \Loc( \{ \Cont(G/H) \mid H\leq G \})
  \]
  be the \emph{localizing subcategory} of $\KK^G$ generated by the
  $G$-C*-algebras $\Cont(G/H)$ of complex functions for all
  subgroups $H\leq G$, with $G$-action induced by that of $G$ on its
  cosets~$G/H$.  That is, $\Cell(G)$ is the smallest
  triangulated subcategory of $\KK^G$ containing all $\Cont(G/H)$
  and closed under forming arbitrary (countable) coproducts.  The
  algebras in $\Cell(G)$ are called \emph{$G$-cell algebras}.
\end{Def}

\begin{Rem}[{\cite[\S3.1]{DEM14}}]
  \label{Rem:bootG}
  The \emph{$G$-equivariant bootstrap category~$\mathfrak B^G$} is
  defined similarly as the localizing subcategory of $\KK^G$
  generated by certain objects.  It is shown in
  \cite[Sec.\,3.1]{DEM14} that~$\mathfrak B^G$ consists precisely of
  those separable $G$-C*-algebras that are equivariantly
  KK-equivalent to a $G$-action on a C*-algebra of Type~I.  If~$G$
  is a finite cyclic group, then $\Cell(G)= \mathfrak B^G$.  For
  more general finite groups, there is only an inclusion
  $\Cell(G)\subseteq \mathfrak B^G$.
\end{Rem}

\begin{Rem} \label{Rem:cpts} By \cite[Prop.\,2.9]{DellAmbrogio14},
  $\Cell(G)$ is a \emph{rigidly-compactly generated} tensor
  triangulated category, in a countable sense of the term (as here
  there are only countable coproducts to work with).  This entails,
  in particular, that its \emph{compact} objects (those $G$-cell
  algebras $A$ such that $\KK^G(A,-)$ commutes with arbitrary
  countable direct sums) coincide with its \emph{rigid} objects
  (those $A$ admitting an inverse $A^\vee$ with respect to the
  tensor product).  It follows that $\Cell(G)^c$, the full
  subcategory of rigid-compact objects, is itself a tensor
  triangulated category.  In addition, $\Cell(G)^c$ is essentially
  small, \emph{rigid} (it admits the self-duality
  $A\mapsto A^\vee$), \emph{idempotent complete} (every idempotent
  endomorphism splits), and it is generated as a thick subcategory
  by the same set of objects:
  \[
    \Cell(G)^c=\Thick (\{\Cont(G/H)\mid H\leq G\}).
  \]
  Recall that a subcategory is \emph{thick} if it is triangulated
  and closed under retracts.
\end{Rem}

\begin{Rem} \label{Rem:funCell}
For finite groups, the usual functors between the equivariant Kasparov categories, such as restriction~$\Res^G_H$, induction~$\Ind^G_H$, and inflation~$\Infl^G_1$, preserve cell algebras and also compact cell algebras.  See \cite[\S2]{DellAmbrogio14}.
\end{Rem}

\section{K\"ohler's universal coefficient theorem}
\label{sec:UCT}%
\medskip

From now on, we restrict attention to $G=\mathbb Z/p\mathbb Z$ for a
prime number~$p$.  For these groups we may use the universal
coefficient theorem (``UCT'') due to K\"ohler \cite{Koehler10} in
order to compute $G$-equivariant KK-theory groups by algebraic
means, at least in principle.  We recall how this works and refer to
\cite{Koehler10} as well as \cite[\S4--5]{Meyer19pp} for all proofs.

Let $\Sigma \Cont(G) \to D \to \unit \to \Cont(G)$ be the mapping
cone sequence (distinguished triangle in $\Cell(G)$) for the unit
map $\unit \to \Cont(G)$, that is, the embedding of~$\mathbb C$ as
constant functions on~$G$.  Let
$B:= \mathbb C \oplus \Cont(G)\oplus D$ be the sum of the three
vertices of the triangle.  Define the $\mathbb Z/2$-graded ring
\[
\mathfrak K := \bigl(\End_{\Cell(G)}(B)_*\bigr)^\op.
\]
Given any $G$-C*-algebra $A\in \KK^G$, its \emph{K\"ohler invariant}
$F_*(A)$ is defined as
\begin{equation}  \label{eq:F_*(A)}
F_*(A)
:= \KK^G_*(B,A)
= \underbrace{ \KK^G_*(\unit, A)}_{
   		\begin{matrix}	F_*(A)_0 \end{matrix} }
\oplus \underbrace{ \KK^G_*(\Cont(G), A)}_{
		\begin{matrix} F_*(A)_1 \end{matrix} }
\oplus  \underbrace{ \KK^G_*(D,A), }_{
 		\begin{matrix} F_*(A)_2  \end{matrix} }
\end{equation}
considered with its evident structure of $\mathbb Z/2$-graded countable left $\mathfrak K$-module.
The assignment $A\mapsto F_*(A)$ extends to a functor $F_*=\KK^G_*(B,-)$ from $\KK^G$ to $\mathbb Z/2$-graded countable left $\mathfrak K$-modules.

K\"ohler's UCT says that there is a natural short exact sequence
\begin{equation} \label{eq:UCT}
\xymatrix@C=13pt{
0\ar[r] & \Ext^1_{\mathfrak K} (F_{*+1}A, F_*A') \ar[r] & \KK^G_*(A,A') \ar[r]^-{F_*}& \Hom_\mathfrak K(F_*A, F_*A') \ar[r]& 0
}
\end{equation}
of $\mathbb Z/2$-graded abelian groups for all $A\in \Cell(G)$ and
$A'\in \KK^G$.  Here $\Hom_\mathfrak K(-,-)$ denotes the graded Hom,
which in degree zero is the group of degree-preserving
$\mathfrak K$-linear maps and in degree one the group of
degree-reversing maps; and similarly for the extension
group~$\Ext^1_\mathfrak K$.  The map denoted by~\(F_*\) is part of the
functor~$F_*$.

\begin{Rem}
  \label{Rem:iso_classif}
  A standard consequence of the UCT is that every isomorphism
  between the K\"ohler invariants $F_*(A)$ and $F_*(B)$ of two
  $G$-cell algebras $A$ and~$B$ lifts to an isomorphism $A\cong B$
  in $\Cell(G)$.  K\"ohler also characterizes the essential image of
  the functor~$F_*$ among the $\mathfrak K$-modules.  Namely, it
  consists precisely of those countable $\mathbb Z/2$-graded
  $\mathfrak K$-modules which are \emph{exact} in the sense that two
  sequences built out of~$M$ and the internal structure
  of~$\mathfrak K$ are exact.  For~$M$ of the form~$F_*(A)$, the two
  sequences are those arising by applying $\KK^G(-,A)$ to the
  triangle linking $\unit, \Cont(G), D$ and its dual triangle under
  Baaj--Skandalis duality.  It follows that $A \mapsto F_*(A)$
  induces a bijection between isomorphism classes of $G$-cell
  algebras and isomorphism classes of exact
  countable $\mathbb Z/2$-graded left $\mathfrak K$-modules.
\end{Rem}

\begin{Rem}
We have displayed in \eqref{eq:F_*(A)} the canonical decomposition of $F_*(A)$ into three parts.
Note that any abstract $\mathfrak K$-module similarly decomposes as
\[
  M = M_0 \oplus M_1 \oplus M_2,
\]
where the three direct summands are the images of the idempotent
elements of~$\mathfrak K$ corresponding to the units of the
endomorphism rings $\End(\unit)$, $\End(\Cont(G))$ and $\End(D)$,
respectively.  The three parts are precisely the terms appearing in
each of the two sequences defining the exactness of $M$ as in
\Cref{Rem:iso_classif}.
\end{Rem}

\begin{Rem}
  \label{Rem:invert_p}
  Both endomorphism rings $\End(\unit)= \Rep(G)$ and
  $\End(\Cont(G))$ are canonically isomorphic to
  $\mathbb Z[x]/(x^p-1)$.  So $M_0$ and $M_1$ are
  $\mathbb Z[x]/(x^p-1)$-modules for any $\mathfrak K$-module~$M$.
  In view of later sections, let us detail what happens if we
  invert~$p$.  The polynomial $x^p-1$ has two irreducible factors,
  $x-1$ and $\Phi_p(x):= 1+x + \cdots +x^{p-1}$.  After
  inverting~$p$, they give rise to the product decomposition
\[
\mathbb Z[x, p^{-1}]/(x^p-1)
\cong \underbrace{\mathbb Z[x, p^{-1}]/(x-1)}_{{} \cong \mathbb Z[p^{-1}]} \times \underbrace{\mathbb Z[x,p^{-1}]/(\Phi_p)}_{{}\cong \mathbb Z[\vartheta, p^{-1}]}
\]
(see \Cref{Rem:Spec-decomp} below for more on this).  Here we write
$\vartheta$ for the `abstract' primitive $p$-th root of unity, that
is, the image of~$x$ in the right-hand side quotient.  Now suppose
that a $\mathfrak K$-module~$M$ is \emph{uniquely $p$-divisible},
that is, it is acted on invertibly by~$p$.  In other words, suppose
it is a $\mathfrak K[p^{-1}]$-module.  Then each of the two parts
$M_0$ and $M_1$ decomposes into a sum of a
$\mathbb Z[p^{-1}]$-module and a
$\mathbb Z[\vartheta, p^{-1}]$-module, via the above actions and
decomposition.
\end{Rem}

\section{A simplification of the bootstrap category}
\label{sec:Q(G)}%
\medskip

Let $G=\mathbb Z/p\mathbb Z$.
We will use the UCT for $G$-cell algebras recalled in~\Cref{sec:UCT}, together with closely related results from~\cite{Meyer19pp}, in order to study the following localization of the $G$-equivariant bootstrap category.

Recall that the \emph{Verdier quotient} of a triangulated category $\mathcal T$ by a thick (\eg, localizing) subcategory $\mathcal S$ is the universal exact functor $\cat T\to \cat T/\cat S$ to a triangulated category $\cat T/\cat S$ in which all objects of $\cat S$ become zero. 
It is obtained as the localization of $\cat T$ which inverts the morphisms whose cone belongs to~$\cat S$; see \cite[\S2]{Neeman01}.

\begin{Not}
  \label{Not:Q}
  Let $\Loc(\Cont(G)) \subset \Cell(G)$ be the localizing
  subcategory generated by the $G$-C*-algebra $\Cont(G)$.  Let
  \[
    Q_G\colon \Cell(G)\longrightarrow \Cell(G)/\Loc(\Cont(G)) =: \cat Q(G)
  \]
  be the Verdier quotient of the bootstrap category $\Cell(G)$ by
  the localizing subcategory generated by the $G$-C*-algebra
  $\Cont(G)$.
\end{Not}

The quotient $\cat Q(G)$  is related to $\Cell(G)$ by a nice
localization sequence of tensor triangulated categories:
\[
\xymatrix{
\Loc(\Cont(G)) \ar[rr]^-{\incl}_{\perp} &&
 \Cell(G) \ar[rr]^-{Q_G}_{\perp} \ar@/^3ex/[ll]^-S &&
  \cat Q(G) \ar@/^3ex/[ll]^-R
}
\]
The following omnibus lemma makes this more precise.

\begin{Lem}
  \label{Lem:Q(G)}
  The quotient functor $Q_G$ admits a coproduct-preserving fully
  faithful right adjoint~$R$.  The inclusion of its full kernel
  $\Ker(Q_G)= \Loc(\Cont(G))$ also admits a coproduct-preserving
  right adjoint~$S$.
  \begin{enumerate}[\rm(1)]
  \item The full essential image of~$R$, which we denote by
    $\cat N:= \Img(R)= \Img(R\circ Q_G)$, is equal to the right
    orthogonal of \(\Loc(\Cont(G))\):
    \[
      \cat N
      = \Loc(\Cont(G))^\perp
      = \Cont(G)^\perp
      := \{ A \in \Cell(G) \mid \KK^G_*(\Cont(G), A) = 0\}.
    \]
  \item The kernel $\Ker(Q_G)= \Loc(\Cont(G))$ equals the left
    orthogonal of~$\cat N$:
    \[
      \Loc(\Cont(G)) = {}^\perp \mathcal N = \{ A \in \Cell(G) \mid \KK^G_* (A,B) =0 \textrm{ for all } B\in \cat N \}.
    \]
  \item There is an essentially unique distinguished triangle in
    $\Cell(G)$ of the form
    \begin{equation} \label{eq:localization_triangle}
      \xymatrix{
        \Sigma N \ar[r] & P \ar[r] & \unit \ar[r] & N,
      }
    \end{equation}
    where $P:= \incl \circ S (\unit)\in \Loc(\Cont(G))$ and
    $N:= R \circ Q_G(\unit)\in \Loc(\Cont(G))^\perp$.
  \item There are isomorphisms of endofunctors of~$\Cell(G)$
    \[
      \incl \circ S \cong P \otimes -
      \quad \textrm{ and } \quad
      R \circ Q_G \cong N\otimes -.
    \]
  \end{enumerate}
\end{Lem}

\begin{proof}
This all follows from standard results on rigidly-compactly generated categories, modulo the fact that we are in the \emph{countably generated} setting.  All claims can be deduced from \cite[\S2]{DellAmbrogio10} for the case $\alpha = \aleph_1$.

More precisely, we begin by noticing that $\Thick(\Cont(G))$ is a tensor ideal in $\Cell(G)^c$ and $\Loc(\Cont(G))$ is a tensor ideal in~$\Cell(G)$.
Both follow from the isomorphisms $\Cont(G)\otimes \Cont(G/H)\cong \bigoplus_{G/H} \Cont(G)$ for all $H\leq G$ and \cite[Lem.\,2.5]{DellAmbrogio14}.

Then we apply \cite[Thm.\,2.28]{DellAmbrogio10} to the rigidly-compactly generated category $\cat T:=\Cell(G)$ (see \Cref{Rem:cpts}) and its tensor ideal of compact objects $\cat J:=\Thick(\Cont(G))$ to conclude that $\cat L:= \Loc(\Cont(G))$ and $\cat L^\perp=\cat N$ form a pair of localizing tensor ideals of~$\Cell(G)$ which are \emph{complementary} as in \cite[Def.\,2.7]{DellAmbrogio10}.
All the remaining claims then follow from \cite[Prop.\,2.26]{DellAmbrogio10}.
\end{proof}

By construction, $\cat Q(G)$ enjoys the same structural properties as~$\Cell(G)$ but with an extra simplification: it is generated by its tensor unit.

\begin{Cor}
  \label{Cor:Q(G)}
  The category $\cat Q(G)$ is a tensor triangulated category that is
  rigidly-compactly generated \textup{(}in the countable
  sense\textup{)}.  It is generated by its tensor unit, that is,
  $\cat Q(G)=\Loc(\unit)$, and also $\cat Q(G)^c=\Thick(\unit)$.
  The quotient functor~$Q_G$ is an exact tensor functor and
  preserves coproducts.  Moreover, $Q_G\colon \Cell(G)\to \cat Q(G)$
  preserves compact objects; the image
  $Q_G(\Cell(G)^c)\subseteq \cat Q(G)^c$ identifies with the Verdier
  quotient $\Cell(G)^c/\Thick(\Cont(G))$ and embeds fully faithfully
  in $\cat Q(G)^c$.  Its image is \emph{dense}, that is, any object
  of $\cat Q(G)^c$ is a retract of one in \(Q_G(\Cell(G)^c)\).
\end{Cor}

\begin{proof}
  Again, these are standard consequences.  Clearly, being the
  quotient of $\Cell(G)$ by a localizing tensor ideal, the category
  $\cat Q(G)$ inherits from $\Cell(G)$ a tensor triangulated
  structure and countable coproducts, and these are preserved by the
  quotient functor~$Q_G$.

  The functor $Q_G$ has a coproduct-preserving right adjoint.
  A short computation using this shows that it preserves compact
  objects.  One verifies similarly that $\cat Q(G)$ is generated by
  the image under $Q_G$ of the compact generators of~$\Cell(G)$.
  Since $Q_G(\Cont(G))\cong 0$ by construction, the compact-rigid
  object $\unit_{\cat Q(G)}= Q_G(\unit_{\Cell(G)})$ suffices.  The
  remaining claims, which are harder, are all part of Neeman's
  localization theorem, in its countable form (see
  \cite[Thm.\,2.10]{DellAmbrogio10}).
\end{proof}

Let us explicitly record another, immediate consequence of \Cref{Lem:Q(G)}:

\begin{Cor} \label{Cor:Endos}
The right adjoint $R$ of $Q_G$ restricts to a canonical equivalence $\cat Q(G)\overset {\sim}{\to} \Img(R)=\cat N$ of triangulated categories, which further restricts to an isomorphism $\End_{\cat Q(G)}(\unit)_* \overset{\sim}{\to}  \End_{\Cell(G)}(N)_*$ of graded endomorphism rings.
\qed
\end{Cor}

In the remainder of this section, we move beyond abstract
generalities.  The next two propositions classify the objects of
$\mathcal Q(G)$ algebraically and compute the graded endomorphism
ring of its unit.  We will actually work within~$\cat N$ (thanks to
\Cref{Cor:Endos}) and use K\"ohler's UCT as well as a refinement due
to the second author (see \cite{Meyer19pp}).  K\"ohler's invariant
$F_*\colon \KK^G\to \mathfrak K\MMod^{\mathbb Z/2}_\infty$ is
recalled in \Cref{sec:UCT}.  Its target is the category of
$\mathbb Z/2$-graded countable left $\mathfrak K$-modules.  Recall
that any $\mathfrak K$-module decomposes canonically as
$M=M_0\oplus M_1\oplus M_2$.

\begin{Lem}
  \label{Lem:alg-class-N_1}
  Suppose that $M\in \mathfrak K\MMod^{\mathbb Z/2}_\infty$ is exact
  as in Remark~\textup{\ref{Rem:iso_classif}} and satisfies $M_1=0$.
  Then~$M$ is uniquely $p$-divisible and is entirely determined
  by~$M_0$ with its natural structure of graded
  $\mathbb
  Z[\vartheta,p^{-1}]$-module.  
  Similarly, every $\mathfrak K$-linear map of such modules is
  uniquely determined by its restriction to their $M_0$-parts.  This
  gives an equivalence
  \[
    \{M \textrm{ exact and }M_1=0\} \overset{\sim}{\longrightarrow}\mathbb Z[\vartheta,p^{-1}]\MMod^{\mathbb Z/2}_\infty  , \qquad M\mapsto M_0,
  \]
  between the full subcategory of exact $\mathfrak K$-modules~$M$
  with $M_1=0$ and the category of $\mathbb Z/2$-graded countable
  $\mathbb Z[\vartheta,p^{-1}]$-modules.
\end{Lem}

\begin{proof}
  Let $M$ be an exact $\mathfrak K$-module whose $M_1$-part
  vanishes.  In particular, and trivially, the
  abelian group $M_1$ is $p$-divisible.  Then
  \cite[Thm.\,7.2]{Meyer19pp} applies to~$M$ and says that the whole
  group~$M$ is uniquely $p$-divisible and that its
  $\mathfrak K$-module structure is of the form described in
  \cite[Ex.\,7.1]{Meyer19pp}.  In particular,
\[
M_0 = X \oplus Y , \qquad
M_1 = X \oplus Z , \qquad
M_2 = Y \oplus \Sigma Z,
\]
where $X$ is some $\mathbb Z/2$-graded $\mathbb Z[p^{-1}]$-module
and $Y,Z$ are some $\mathbb Z/2$-graded
$\mathbb Z[\vartheta,p^{-1}]$-modules.  The decompositions
of $M_0$ and~$M_1$ arise from the action of $\mathfrak K[p^{-1}]$ as
in \Cref{Rem:invert_p}.  Since $M_1=0$ in our case, it follows that
$X=Z=0$.  So
\begin{equation}
  \label{eq:special_form_module}
  M_0 = Y, \qquad
  M_1 = 0 , \qquad
  M_2 = Y.
\end{equation}
The construction in \cite[Ex.\,7.1]{Meyer19pp} shows that the whole
$\mathfrak K[p^{-1}]$-action on such a $\mathfrak K$-module~$M$ is
determined by the $\mathbb Z[\vartheta,p^{-1}]$-action on $Y=M_0$.
Similarly, every $\mathbb Z[x]/(x^p-1)$-linear map $Y\to Y'$ admits
a unique $\mathfrak K$-linear extension $M\to M'$ to the
corresponding $\mathfrak K$-modules.  This proves the lemma.
\end{proof}

\begin{Prop}
  \label{Prop:alg-class-N_2}
  The functor $F_*$ restricts between $\cat N$ and the full
  subcategory of $\mathfrak K\MMod^{\mathbb Z/2}_\infty$ of those exact $M$ such that $M_1=0$ as in Lemma~\textup{\ref{Lem:alg-class-N_1}}.  In particular,
  the $G$-equivariant K-theory functor
  \[
    A\mapsto F_*(A)_0 := \KK^G_*(\unit, A)
  \]
  induces a bijection between the isomorphism classes of objects in
  $\cat N$ and those in
  $\mathbb Z[\vartheta,p^{-1}]\MMod^{\mathbb Z/2}_\infty$.
\end{Prop}

\begin{proof}
We know from K\"ohler's classification
(\Cref{Rem:iso_classif}) that $F_*$ induces a bijection between the
isomorphism classes of $G$-cell algebras and those of exact 
countable graded modules.  If $A\in \mathcal N =
\Loc(\Cont(G))^\perp$, then
\begin{equation} \label{eq:}
F_*(A)_1 := \KK^G_*(\Cont(G),A)=0.
\end{equation}
So $F_*(A)$ is a $\mathfrak K$-module as in \Cref{Lem:alg-class-N_1}.  Hence the
functor $F_*$ restricts as claimed.  Moreover, any (graded
countable) $\mathbb Z[\vartheta,p^{-1}]$-module gives rise to a
unique exact (countable \(\mathbb Z/2\)-graded) $\mathfrak K$-module of the form
\eqref{eq:special_form_module}.  Therefore, K\"ohler's
classification combines with \Cref{Lem:alg-class-N_1} to yield the claimed
classification for~$\cat N$.
\end{proof}

We conclude the section with this central computation:

\begin{Prop}
  \label{Prop:Kth-comput}
  As above, let $G=\mathbb Z/p\mathbb Z$ for a prime number~$p$.
  The graded endomorphism ring of the tensor unit
  $\unit \in \cat Q(G)$ is given by
  \[
    \End_{\cat Q(G)}(\unit)_* \cong \mathbb Z[\vartheta, p^{-1}, \beta^{\pm 1}],
  \]
  where~$\vartheta$ is a primitive $p$-th root of unity \textup(in
  the sense of Remark~\textup{\ref{Rem:invert_p}} and set in degree
  zero\textup) and where~$\beta$ is the invertible Bott element
  \textup(in degree two\textup).  More precisely, the restriction
  $\End_{\Cell(G)}(\unit)_*\to \End_{\cat Q(G)}(\unit)_*$ of the
  localization functor~$Q_G$ identifies with the canonical
  grading-preserving ring map which inverts $p$ and kills the ideal
  generated by $\Phi_p(x)=1+x+\cdots + x^{p-1}$:
  \[
    \mathbb Z[x]/(x^p-1)[\beta^{\pm 1}] \longrightarrow \mathbb Z[x, p^{-1}]/(\Phi_p)\ [\beta^{\pm 1}] = \mathbb Z[\vartheta, p^{-1}, \beta^{\pm 1}] .
  \]
\end{Prop}

\begin{proof}
Recall the algebra $P \in \Loc(\Cont(G))$ of \Cref{Lem:Q(G)}\,(3).
For any $A$ belonging to $\cat N = \Loc(\Cont(G))^\perp$ we must have
\begin{equation}
  \label{eq:vanishings}
  \KK^G_*(P,A)=0 \quad \textrm{ and } \quad
  F_*(A)_1 := \KK^G_*(\Cont(G),A)=0 .
\end{equation}
The first vanishing group implies that the map $\unit \to N$ of the distinguished triangle~\eqref{eq:localization_triangle} induces a natural isomorphism
\begin{equation} \label{eq:first-nat-iso}
\KK^G_*(N,A) \overset{\sim}{\to} \KK^G_*(\unit , A) = F_*(A)_0
\end{equation}
for all such $A\in \cat N$.
Specializing this to the case $A=N$, we see that
\[
\KK^G_*(N,N) \cong F_*(N)_0 .
\]
We claim that $F_*(N)_0$, and therefore $\KK^G_*(N,N)$, is a free $\mathbb Z[\vartheta,p^{-1}]$-module of rank one and concentrated in degree zero.
 By \Cref{Cor:Endos} and after unwinding $\mathbb Z/2$-gradings by Bott periodicity, this would prove the proposition.

By \Cref{Prop:alg-class-N_2}, there is an object $R\in \mathcal N$ such that $F_*(R)_0$ (and thus $F_*(R)_2$) is a free $\mathbb Z[\vartheta,p^{-1}]$-module of rank one concentrated in degree zero.
To prove the claim, it now suffices to prove that $N$ and $R$ are isomorphic in~$\cat N$.

Consider the UCT exact sequence \eqref{eq:UCT} for an arbitrary object $A\in \cat N$:
\begin{equation} \label{eq:UCT_specialized}
\xymatrix@C=13pt{
0\ar[r] & \Ext^1_{\mathfrak K} (F_{*+1}R, F_*A) \ar[r] & \KK^G_*(R,A) \ar[r]^-{F_*}& \Hom_\mathfrak K(F_*R, F_*A) \ar[r]& 0
}.
\end{equation}
By construction, $F_*(R)_0$ is a free
$\mathbb Z[\vartheta, p^{-1}]$-module of rank one in degree zero.
Hence by \Cref{Lem:alg-class-N_1}, the Hom-term in
\eqref{eq:UCT_specialized} reduces to
\[
\Hom_{\mathbb Z[\vartheta, p^{-1}]} \bigl( F_*(R)_0, F_*(A)_0 \bigr)
\cong F_*(A)_0.
\]
Now, suppose that the Ext-term in \eqref{eq:UCT_specialized}
vanishes.  We would obtain an isomorphism
\[
F_*\colon \KK^G_*(R,A) \overset{\sim}{\to}  F_*(A)_0,
\]
natural in $A\in \cat N$.  By combining it with the isomorphism
\eqref{eq:first-nat-iso} and applying the Yoneda Lemma for the
category~$\cat N$, this would show that $R\cong N$ as wished.  Thus
it only remains to show that the Ext-term of
\eqref{eq:UCT_specialized} vanishes.

Consider an arbitrary extension
\begin{equation} \label{eq:K-ext}
\xymatrix{
0 \ar[r] & F_*(A) \ar[r] & M \ar[r] & F_{*+1} (R) \ar[r] & 0
}
\end{equation}
of graded $\mathfrak K$-modules.  By applying to \eqref{eq:K-ext} the idempotent element of $\mathfrak K$ corresponding to the identity map of~$\Cont(G)$ we obtain an exact sequence

\[
\xymatrix{
0 \ar[r] & F_*(A)_1 \ar[r] & M_1 \ar[r] & F_{*+1} (R)_1 \ar[r] & 0
}
\]
where both outer terms vanish by hypothesis.  Then $M_1=0$.  In
particular, $M_1$ is uniquely $p$-divisible.  Since~$M$ is exact as
an extension of two exact $\mathfrak K$-modules, we may apply
\cite[Thm.\,7.2]{Meyer19pp} to it.  Thus~$M$ also has the special
form of \Cref{Lem:alg-class-N_1} and is uniquely determined by
its $M_0$-part, viewed as a $\mathbb Z[\vartheta,p^{-1}]$-module.
Now consider the extension of graded
$\mathbb Z[\vartheta,p^{-1}]$-modules
\[
\xymatrix{
0 \ar[r] & F_*(A)_0 \ar[r] & M_0 \ar[r] & F_{*+1} (R)_0 \ar[r] & 0
}
\]
obtained by hitting \eqref{eq:K-ext} with the idempotent of
$\mathfrak K$ corresponding to the identity of~$\unit$.  This
extension must split because $F_{*+1} (R)_0$ is a free module.
Moreover, by \Cref{Lem:alg-class-N_1} once again, any
$\mathbb Z[\vartheta,p^{-1}]$-linear section of
$M_0 \to F_{*+1} (R)_0$ extends to a $\mathfrak K$-linear section of
$M\to F_{*+1} (R)$.  Thus the original extension \eqref{eq:K-ext} of
$\mathfrak K$-modules splits as well.  As the latter extension was
arbitrary, this implies that
$\Ext^1_{\mathfrak K} (F_{*+1}R, F_*A) =0 $ as required.
\end{proof}

\section{The Balmer spectrum}
\label{sec:ttg}%
\medskip

We very briefly recall some basic notions of tensor triangular
geometry, referring to \cite{BalmerICM} and the original references
therein for more details.

Let~$\cat K$ be an essentially small tensor triangulated category,
with tensor $\otimes$ and unit object~$\unit$.  Its \emph{spectrum}
$\Spc \cat K$ is the set of its \emph{prime
  $\otimes$-ideals}~$\mathcal P$, that is, those proper, full and
thick subcategories $\cat P \subsetneq \cat K$ which are prime
tensor ideals for the tensor product: $A \otimes B \in \cat P$
if and only if $A\in \cat P$ or $B\in \cat P$, for any objects
$A,B\in \cat K$.  The spectrum is endowed with the `Zariski'
topology, which has the family of subsets
\begin{equation} \label{eq:supp}
\supp (A) := \{\cat P\in \Spc \cat K \mid A \notin \cat P \}
\qquad (A \in \cat K)
\end{equation}
as a basis of closed subsets.

The ring $\End_\cat K(\unit)$ is commutative.  So we can consider
the usual Zariski spectrum of its prime ideals,
$\Spec \End_\cat K(\unit)$.  The assignment
\[
  \cat P \mapsto \rho_\cat K(\cat P)
  :=\{ f\in \End_\cat K(\unit)\mid \cone(f)\not\in \cat P \}
\]
defines a continuous map between the two spectra:
\begin{equation}
  \label{eq:comparison}
  \rho_\cat K\colon \Spc \cat K \longrightarrow \Spec \End_\cat K(\unit)
\end{equation}
(see \cite[Thm.\,5.3]{Balmer10b}).  In general, this comparison map
is neither injective nor surjective.  Surjectivity is more common
and often much easier to prove.

\begin{Rem} \label{Rem:funSpc}
The spectrum is functorial: Any exact tensor functor $F\colon \cat K \to \cat L$ defines a continuous map $\Spc F\colon \Spc \cat L\to \Spc \cat K$ sending a prime $\cat P$ of $\cat L$ to $(\Spc F )(\cat P) := F^{-1}\cat P$.  Moreover, $\Spc (F_2 \circ F_1)= \Spc F_1 \circ \Spc F_2$ and $\Spc \Id = \Id$.
\end{Rem}

\begin{Rem} \label{Rem:rho-nat}

  The comparison map \eqref{eq:comparison} is natural (see
  \cite[Thm.\,5.3\,(c)]{Balmer10b}); that is, if
  $F\colon \cat K\to \cat L$ is an exact tensor functor, then the
  square
  \[
    \xymatrix{
      \Spc \cat L \ar[rr]^-{\Spc F} \ar[d]_{\rho_\cat K} &&
      \Spc \cat K \ar[d]^{\rho_\cat L}  \\
      \Spec \End_\cat L(\unit) \ar[rr]^-{\Spec F} &&
      \Spec \End_\cat K(\unit)
    }
  \]
  commutes; here the bottom arrow is the continuous map
  $\mathfrak p\mapsto F^{-1}\mathfrak p$ induced between Zariski
  spectra by the ring homomorphism
  $F\colon \End_\cat K(\unit)\to \End_\cat L(\unit)$.
\end{Rem}

\section{Computation of the spectrum}
\label{sec:spc}%
\medskip

Finally, we are ready to prove \Cref{Thm:main}.

Let $G= \mathbb Z/p\mathbb Z$ as before.  Let
$\Cell(G)\subset \KK^G$ be the tensor triangulated category of
$G$-cell algebras (\Cref{Def:Cell}).  Let
$\Cell(G)^c\subset \Cell(G)$ be the tensor triangulated subcategory
of its compact objects (\Cref{Rem:cpts}).  Write
\[
\rho_G := \rho_{\Cell(G)^c} \colon \Spc  \Cell(G)^c \to \Spec \Rep(G)
\]
for the canonical map \eqref{eq:comparison} comparing the triangular spectrum of $\Cell(G)^c$ to the Zariski spectrum of the representation ring $\Rep(G)=\End_{\Cell(G)^c}(\unit) = \mathbb Z[x]/(x^p-1)$.

Our goal is to show that $\rho_G$ is a homeomorphism.
To this end, we will adapt the `divide and conquer' strategy of~\cite{BalmerSanders17}  to the two exact tensor functors
\begin{equation} \label{eq:decomp_functors}
\xymatrix{
\Cell(1) & \Cell(G) \ar[l]_-{\Res^G_1} \ar[r]^-{Q_G} & \cat Q(G)
}
\end{equation}
given by restriction to the trivial group, $\Res^G_1$, and the
Verdier quotient functor~$Q_G$ studied in \Cref{sec:Q(G)}.

\begin{Lem}
  \label{Lem:reg-noeth}
  The graded commutative ring
  $
  \End_{\cat Q(G)}(\unit)_* \cong \mathbb Z[\vartheta, p^{-1}, \beta^{\pm 1}]
  $
  computed in Proposition~\textup{\ref{Prop:Kth-comput}} is
  Noetherian and regular.
\end{Lem}

\begin{proof}
  The degree zero subring $\mathbb Z[\vartheta, p^{-1}]$ is
  Noetherian because it is finitely generated and commutative.  It
  is regular because it is a localization of the classical Dedekind
  domain
  $\mathbb Z[\vartheta]\cong \mathbb Z[\zeta_p]\subset \mathbb C$,
  where~$\zeta_p$ is a complex primitive $p$-th root of unity.
  It follows that our graded ring is also Noetherian and regular.
\end{proof}

\begin{Rem}
  \label{Rem:Spech-vs-Spec}
  If~$R_*$ is a $\mathbb Z$-graded commutative ring, we can consider
  the Zariski spectrum $\Spech R_*$ of its \emph{homogeneous} prime
  ideals.  The inclusion of the zero degree subring $R_0\subset R_*$
  induces a continuous restriction map $\Spech R_* \to \Spec R_0$.
  If the graded ring $R_*$ is 2-periodic and concentrated in even
  degrees --~that is, $R_* = R_0 [\beta^{\pm 1}]$ with
  $|\beta|=2$~-- the latter map is easily seen to be a homeomorphism
  $\Spech R_* \overset{\sim}{\to} \Spec R_0$.  Of course, this
  applies to our endomorphism rings.
\end{Rem}

\begin{Rem}
  \label{Rem:Spec-decomp}
  Consider the Zariski spectrum of $\Rep(G)=\mathbb Z[x]/(x^p-1)$.
  It has two irreducible components, namely, the images of the two
  embeddings
\[
\xymatrix{
\Spec \mathbb Z \ar[r]^-{\psi} & \Spec \mathbb Z[x]/(x^p-1)  & \Spec \mathbb Z[x]/(\Phi_p) \ar[l]_-{\varphi}
}
\]
induced by the two ring quotients
\[
\xymatrix{
\mathbb Z & \mathbb Z[x]/(x^p-1) \ar[l] \ar[r] & \mathbb Z[x]/(\Phi_p)
}
\]
given by killing $x-1$ and $\Phi_p=1 + x + \cdots + x^{p-1}$,
respectively; these are the two irreducible factors of~$x^p-1$
in~$\mathbb Z[x]$.  Thus the maps $\psi$ and $\varphi$ are jointly
surjective.  The intersection of their images can be shown to
consist exactly of one point lying above~$p$, namely, the maximal
ideal $\psi((p))=(p,x-1)=(p,\Phi_p)=\varphi((p))$.  By inverting $p$
in $\mathbb Z[x]/(\Phi_p)$ we get rid of precisely the preimage
under $\varphi$ of this common point, thus eliminating the
redundancy.  In particular, we obtain a decomposition
\[
\Spec \Rep(G) \cong \Spec \mathbb Z \sqcup \Spec \mathbb Z[x, p^{-1}]/(\Phi_p)
\]
as sets.  As before, we write
$\mathbb Z[\vartheta, p^{-1}] := \mathbb Z[x, p^{-1}]/(\Phi_p)$.
(This is all well-known; see, for instance, \cite{BuenoDokuchaev08}
and the references therein for explanations and context.)
\end{Rem}

Now we know everything we need about the geometry of~$Q_G$.
As for the restriction functor $\Res^G_1\colon \Cell(G)^c\to \Cell(1)^c$, we need the following:

\begin{Lem}  \label{Lem:sepSpc}
The equality $\supp \Cont(G)= \Img(\Spc \Res^G_1)$, between the support of the object $\Cont(G)$ and the image of the map $\Spc(\Res^G_1)$, holds  in $\Spc \Cell(G)^c$.
\end{Lem}

\begin{proof}
  By \cite[Thm.\,1.2]{BalmerDellAmbrogioSanders15} (see also
  \cite[\S2.4]{BalmerDellAmbrogio20pp}), the restriction functor
  $\Res^G_1$ is a \emph{finite separable extension}.  More
  precisely, this is proved in \cite{BalmerDellAmbrogioSanders15} for the restriction functor
  $\KK^G\to \KK(1)$ between the \emph{whole} Kasparov categories.
  But the result also holds for (compact) cell algebras.  Let us
  briefly recall this.  The two-sided adjunction between $\Ind^G_1$
  and $\Res^G_1$ restricts to a two-sided adjunction between
  $\Cell(G)$ and $\Cell(1)$ (\Cref{Rem:funCell}), and provides us
  with a separable commutative monoid $A^G_1$ in $ \Cell(G)$ whose
  underlying object is $\Ind^G_1(\unit)=\Cont(G)$.  Exactly the same
  proofs as in \cite{BalmerDellAmbrogioSanders15} yield a canonical
  equivalence of tensor triangulated categories
  \[
    \Cell(1) \simeq A^G_1\MMod_{\Cell(G)}
  \]
  between the bootstrap category $\Cell(1)$ and the Eilenberg--Moore
  category of modules in $\Cell(G)$ over the monoid~$A^G_1$.  This
  equivalence identifies $\Res^G_1$ with the `free module' functor
  $F:=A^G_1\otimes-\colon \Cell(G)\to A^G_1\MMod_{\Cell(G)}$.  It may also be restricted to compact objects:
  $\Cell(1)^c\simeq (A^G_1\MMod_{\Cell(G)})^c=
  A^G_1\MMod_{\Cell(G)^c}$.

As with any separable monoid, \cite[Thm.\,1.5]{Balmer16} shows the equality $\Img(\Spc F)= \supp A^G_1$.  Hence
$\Img (\Spc \Res^G_1)=\supp \Cont(G)$ by the above identifications.
\end{proof}

\begin{proof}[Proof of Theorem~\textup{\ref{Thm:main}}]
  We already know from \cite[Thm.\,1.4]{DellAmbrogio10} that the map
  $\rho_G$ admits a continuous section --~even for any finite
  group~$G$.  To show that it is a homeomorphism, it will therefore
  suffice to prove its injectivity.

We claim that the two functors \eqref{eq:decomp_functors} induce the following commutative diagram:
\begin{equation} \label{eq:decomp_maps}
\vcenter{
\xymatrix{
\Spc \Cell(1)^c \ar[d]_{\rho_1} \ar[rr]^-{\Spc \Res^G_1} &&
 \Spc \Cell(G)^c  \ar[d]_{\rho_G}  &&
  \Spc \cat Q(G)^c \ar[ll]_-{\Spc Q_G} \ar[d]_{\rho_{\cat Q(G)^c} \,=:\, \rho_\cat Q} \\
\Spec \mathbb Z \ar[rr] \ar[rr]^-{\psi} &&
 \Spec \mathbb Z[x] / (x^p-1) &&
  \Spec \mathbb Z[\vartheta,p^{-1}] \ar[ll]_-{\varphi}
}}
\end{equation}
Indeed, the top row is obtained by restricting the two functors to
rigid-compact objects and applying the functoriality of $\Spc(-)$
(\Cref{Rem:funSpc}).  The three vertical maps are all instances of
the canonical comparison~\eqref{eq:comparison}.  The two squares
commute because the latter is natural (see \Cref{Rem:rho-nat}).
The bottom row is as in \Cref{Rem:Spec-decomp}: Indeed, the right
arrow is given by inverting $p$ and killing~$\Phi_p$ (by
\Cref{Prop:Kth-comput}) and the left arrow by mapping $x\mapsto 1$
(because it corresponds to the rank homomorphism
$\Rep(G)\to \Rep(1)=\mathbb Z$).

By \Cref{Rem:Spec-decomp}, the bottom row of \eqref{eq:decomp_maps} is a disjoint-union decomposition of the set $\Spec \mathbb Z[x]/(x^p-1)$.

We claim that the top row of \eqref{eq:decomp_maps} provides a
similar decomposition, that is, it consists of two injective and
jointly surjective maps with disjoint images, so that
\begin{eqnarray} \label{eq:top-decomp}
\Spc \Cell(G)^c = \Img( \Spc \Res^G_1) \sqcup \Img( \Spc Q_G)  \,.
\end{eqnarray}
Indeed, we obviously have
\[
\Spc \Cell(G)^c =   \{\mathcal P \mid \Cont(G)\notin \mathcal P\}  \sqcup \{\mathcal P\mid \Cont(G)\in \mathcal P\}\,.
\]
Moreover, by \Cref{Cor:Q(G)} the restriction of $Q_G$ on compact objects factors as a Verdier quotient with kernel $\Thick(\Cont(G))$ followed by a full dense embedding:
\[
\Cell(G)^c \longrightarrow \Cell(G)^c/\Thick(\Cont (G)) \hooklongrightarrow \cat Q(G)^c .
\]
By basic tensor-triangular results \cite[Propositions~3.11 and~3.13]{Balmer05a}, the induced map $\Spc Q_G$ is therefore injective with image
\[
\Img( \Spc Q_G) = \{\mathcal P\mid \Thick(\Cont(G)) \subseteq \mathcal P\} = \{\mathcal P\mid \Cont(G)\in \mathcal P\} \,.
\]
Let us consider the other half of the decomposition. 
Recall the inflation functor $\Infl^G_1\colon \Cell(1)^c\to \Cell(G)^c$ (\Cref{Rem:funCell}),  which endows each
$A\in \Cell(1)^c$ with the trivial $G$-action.
It is an exact tensor functor such that $\Res^G_1\circ \Infl^G_1= \Id$. 
Since \(\Spc\) is a contravariant functor (\Cref{Rem:funSpc}), we deduce that
\[
\Spc (\mathrm{Inf}^G_1)\circ \Spc (\Res^G_1) = \Spc (\Res^G_1 \circ \mathrm{Inf}^G_1) = \Spc (\Id_{\Cell(1)^c}) = \Id.
\]
Thus $\Spc (\Res^G_1)$ is injective. 
By  \Cref{Lem:sepSpc} and the definition of support~\eqref{eq:supp}, its image is
\[
\Img (\Spc \Res^G_1) = \supp \Cont(G) =  \{\mathcal P \mid \Cont(G)\notin \mathcal P\} .
\]
This concludes the proof of the decomposition~\eqref{eq:top-decomp}.

The above decompositions of the triangular and Zariski spectra and
the commutative diagram \eqref{eq:decomp_maps} imply that the middle
vertical map~$\rho_G$ is bijective if and only if both $\rho_1$
and~$\rho_{\cat Q}$ are.  We know from
\cite[Thm.\,1.2]{DellAmbrogio10} that~$\rho_1$ is bijective.  As
for~$\rho_{\cat Q(G)}$, we will appeal to
\cite{DellAmbrogioStanley16}.

The comparison map has a graded version
\[
\rho^*_{\cat Q} \colon \Spc \cat Q(G)^c \longrightarrow \Spech \End(\unit)_*,
\]
whose target is the \emph{homogeneous} spectrum of the \emph{graded}
endomorphism ring of~$\cat Q(G)$ (see \Cref{Rem:Spech-vs-Spec});
this follows from \cite[Thm.\,5.3]{Balmer10b} for the choice
$u=\Sigma(\unit)$ of grading object.  By \Cref{Cor:Q(G)},
$\cat Q(G)^c$ is generated by its tensor unit as a thick
subcategory.  By \Cref{Lem:reg-noeth}, the graded ring
$\End_{\cat Q(G)^c}(\unit)_*$ is Noetherian and regular.  Therefore,
\cite[Thm.\,1.1]{DellAmbrogioStanley16} shows that the graded
comparison map $\rho^*_{\cat Q}$ is bijective.  The
map~$\rho_{\cat Q}$ is bijective as well because the isomorphism
$\Spech \End(\unit)_* \cong \Spec \End(\unit)$ of
\Cref{Rem:Spech-vs-Spec} identifies $\rho_G^*$ with~$\rho_Q$ (see
\cite[Cor.\,5.6.(b)]{Balmer10b}).  This completes the proof.
\end{proof}


\begin{Rem} \label{Rem:cf-SH1}
As already mentioned, the present proof of \Cref{Thm:main} is loosely inspired by the analogous determination (as a set) of the spectrum of~$\SH(G)^c$, the stable homotopy category of compact $G$-spectra; more precisely, by the proof of \cite[Thm.\,4.9]{BalmerSanders17}.
The latter argument works for any finite group $G$ by induction on its order, and this induction \emph{could} be adapted to yield a homeomorphism $\rho\colon \Spc \Cell(G)^c\overset{\sim}{\to} \Spec \Rep(G)$ for general~$G$, \emph{provided} we knew that a certain ring is regular for all~$G$ (in order to invoke~\cite{DellAmbrogioStanley16}).
The ring in question is the graded endomorphism ring of the tensor unit in the tensor triangulated category
\[
\Cell(G)/ \Loc(\{\Cont(G/H)\mid H\lneq G\}),
\]
which we currently do not know how to compute.  In other words, we would need to find a general replacement for our use of K\"ohler's UCT in \Cref{sec:Q(G)}.
\end{Rem}

\begin{Rem} \label{Rem:cf-SH2}
Besides the proofs' analogies, the result in \cite{BalmerSanders17} is actually quite different from ours.  Most strikingly, the comparison map $\rho_{\SH(G)^c}$ is very far from being injective.  In particular, \cite{DellAmbrogioStanley16} cannot be applied to the case of $G$-spectra; in this case, the role of \cite{DellAmbrogioStanley16} in the proof's structure is played instead by the fact that the composite of inflation followed by localization
\[
\xymatrix{
\SH \ar[r]^-{\Infl_1^G} & \SH(G) \ar[r] & \SH(G)/\Loc(\{\Sigma^\infty_+ G/H \mid H\lneq G\})
}
\]
is an equivalence of tensor triangulated categories (see \cite[\S2\,(H)]{BalmerSanders17}).
The analogous result for KK-theory is false.  By \Cref{Prop:Kth-comput}, it fails for $G\cong \mathbb Z/p\mathbb Z$.

\end{Rem}

\bibliographystyle{alpha}

\begin{thebibliography}{DEM14}

\bibitem[Bal05]{Balmer05a}
Paul Balmer.
\newblock The spectrum of prime ideals in tensor triangulated categories.
\newblock {\em J. Reine Angew. Math.}, 588:149--168, 2005.

\bibitem[Bal10a]{Balmer10b}
Paul Balmer.
\newblock Spectra, spectra, spectra -- tensor triangular spectra versus
  {Z}ariski spectra of endomorphism rings.
\newblock {\em Algebr. Geom. Topol.}, 10(3):1521--1563, 2010.

\bibitem[Bal10b]{BalmerICM}
Paul Balmer.
\newblock Tensor triangular geometry.
\newblock In {\em International {C}ongress of {M}athematicians, Hyderabad
  (2010), {V}ol. {II}}, pages 85--112. Hindustan Book Agency, 2010.

\bibitem[Bal16]{Balmer16}
Paul Balmer.
\newblock Separable extensions in tensor-triangular geometry and generalized
  {Q}uillen stratification.
\newblock {\em Ann. Sci. \'Ec. Norm. Sup\'er. (4)}, 49(4):907--925, 2016.

\bibitem[Bal19]{Balmer19}
Paul Balmer.
\newblock A guide to tensor-triangular classification.
\newblock In {\em Handbook of Homotopy Theory}, pages 147--164. Chapman and
  Hall/CRC, 2019.

\bibitem[BD08]{BuenoDokuchaev08}
Andre~Gimenez Bueno and Michael Dokuchaev.
\newblock On spectra of abelian group rings.
\newblock {\em Publ. Math. Debrecen}, 72(3-4):269--284, 2008.

\bibitem[BD20]{BalmerDellAmbrogio20pp}
Paul Balmer and Ivo Dell'Ambrogio.
\newblock {\em {M}ackey 2-functors and {M}ackey 2-motives}.
\newblock EMS Monographs in Mathematics. European Mathematical Society (EMS),
  Z\"{u}rich, 2020.

\bibitem[BDS15]{BalmerDellAmbrogioSanders15}
Paul Balmer, Ivo Dell'Ambrogio, and Beren Sanders.
\newblock Restriction to finite-index subgroups as \'etale extensions in
  topology, {KK}-theory and geometry.
\newblock {\em Algebr. Geom. Topol.}, 15(5):3025--3047, 2015.

\bibitem[BS17]{BalmerSanders17}
Paul Balmer and Beren Sanders.
\newblock The spectrum of the equivariant stable homotopy category of a finite
  group.
\newblock {\em Invent. Math.}, 208(1):283--326, 2017.

\bibitem[Del10]{DellAmbrogio10}
Ivo Dell'Ambrogio.
\newblock Tensor triangular geometry and {$KK$}-theory.
\newblock {\em J. Homotopy Relat. Struct.}, 5(1):319--358, 2010.

\bibitem[Del11]{DellAmbrogio11}
Ivo Dell'Ambrogio.
\newblock Localizing subcategories in the bootstrap category of separable
  {$C^\ast$}-algebras.
\newblock {\em J. K-Theory}, 8(3):493--505, 2011.

\bibitem[Del14]{DellAmbrogio14}
Ivo Dell'Ambrogio.
\newblock Equivariant {K}asparov theory of finite groups via {M}ackey functors.
\newblock {\em J. Noncommut. Geom.}, 8(3):837--871, 2014.

\bibitem[DEM14]{DEM14}
Ivo Dell'Ambrogio, Heath Emerson, and Ralf Meyer.
\newblock An equivariant {L}efschetz fixed-point formula for correspondences.
\newblock {\em Doc. Math.}, 19:141--194, 2014.

\bibitem[DS16]{DellAmbrogioStanley16}
Ivo Dell'Ambrogio and Donald Stanley.
\newblock Affine weakly regular tensor triangulated categories.
\newblock {\em Pacific J. Math.}, 285(1):93--109, 2016.

\bibitem[Kas88]{Kasparov88}
G.~G. Kasparov.
\newblock Equivariant {$KK$}-theory and the {N}ovikov conjecture.
\newblock {\em Invent. Math.}, 91(1):147--201, 1988.

\bibitem[K{\"o}h10]{Koehler10}
Manuel K{\"o}hler.
\newblock Universal coefficient theorems in equivariant {KK}-theory.
\newblock Ph.D.\, Thesis, Georg-August-Universit\"at G\"ottingen. Available at
  \texttt{http://hdl.handle.net/11858/00-1735-0000-0006-B6A9-9}, 2010.

\bibitem[Mey21]{Meyer19pp}
Ralf Meyer.
\newblock On the classification of group actions on \(\mathrm{C}^*\)-algebras
  up to equivariant {KK}-equivalence.
\newblock {\em Ann. K-Theory}, 6(2):157--238, 2021. DOI 10.2140/akt.2021.6.157

\bibitem[MN06]{MeyerNest06}
Ralf Meyer and Ryszard Nest.
\newblock The {B}aum-{C}onnes conjecture via localisation of categories.
\newblock {\em Topology}, 45(2):209--259, 2006.

\bibitem[Nee01]{Neeman01}
Amnon Neeman.
\newblock {\em Triangulated categories}, volume 148 of {\em Annals of
  Mathematics Studies}.
\newblock Princeton University Press, 2001.

\end{thebibliography}

\vspace{-.19cm} 

\printindex
\end{document}